\documentclass[12pt, oneside]{amsart}
\usepackage[all,cmtip]{xy}
\usepackage[utf8x]{inputenc}
\usepackage{amsmath, amsthm, amssymb}
\usepackage{eucal, mathrsfs, yfonts}
\usepackage{mathtools}
\usepackage{amsfonts}
\usepackage{array}
\usepackage{hyperref}
\usepackage{graphicx}
\usepackage{caption}
\usepackage{pdfpages}
\usepackage{verbatim}
\usepackage{float}
\usepackage[english]{babel}
\usepackage{epstopdf}

\usepackage{geometry}                		% See geometry.pdf to learn the layout options. There are lots.
\newtheorem{theorem}{Theorem}[section]
\theoremstyle{plain}

\newtheorem{lemma}[theorem]{Lemma}

\theoremstyle{remark}

\numberwithin{equation}{section}

\newtheorem*{remark}{Remark}

\newcommand{\zetaa}{\zeta_{\Phi,0}(s)}
\newcommand{\zetab}{\zeta_{\Phi,g_0}(s)}

\begin{document}

\title[]{A connection between discrete and regularized Laplacian determinants on fractals}

\author[Konstantinos Tsougkas]{Konstantinos Tsougkas}

\address{Konstantinos Tsougkas\\
         Department of Engineering\\
         University of Skövde\\
         54128 Skövde\\
         Sweden.}
        
\email{konstantinos.tsougkas@his.se}

\maketitle

\begin{abstract}
The spectral zeta function of the Laplacian on self-similar fractal sets has been previously studied and shown to meromorphically extend to the complex plane. In this work we establish under certain conditions a relationship between the logarithm of the determinant of the discrete graph Laplacian on the sequence of graphs approximating the fractal and the regularized determinant which is defined via help of the spectral zeta function. We then at the end present some concrete examples of this phenomenon.

\end{abstract}

\section{Introduction} \noindent

There has been considerable work on studying the spectrum of the Laplacian on fractals, see \emph{e.g.\@} \cite{chen2015spectral,chen2016singularly,fukushima1992spectral,kigami1993weyl,shima1991eigenvalue,shima1996eigenvalue,fractafold,strichartz2012exact,strichartz2012spectral}. After having an understanding of the spectrum, the spectral zeta function of the Laplace operator has then been studied in \cite{grabner,teplstei,teplyaev2004spectral,teplyaev} and was meromorphically extended to the entire complex plane. Its poles, also called \emph{complex dimensions} \cite{lapidus2000fractal}, are studied in \cite{teplstei} and it is proven that for a large class of p.c.f. fractals under certain symmetry conditions, that the locations of the poles can only be on the imaginary axis or on the axis where ${\rm Re}(s)=\frac{d_S}{2}$, where $d_S$ stands for the spectral dimension. There is an interesting connection between the Hausdorff, spectral and walk dimensions explored further in \cite{freiberg2007einstein}.

Research in theoretical physics such as
 \cite{akkermans2013statistical,
	akkermans2012wave,
akkermans2009physical, 
	akkermans2010thermodynamics, akkermans2013spontaneous,
	dunne2012heat,
    elizalde2012introduction,
	lauscher2005fractal,
	reuter2011fractal,tanese2014fractal} and in more classical works \cite{elizalde1994zeta,englert1987metric,hawking1977zeta,knizhnik1988fractal} has focused on spectral zeta functions and regularized determinants. Our mathematical motivation in this paper comes from the following two parts. Firstly, it comes from the following sentence as stated in \cite{grabner} and \cite{derfel2012laplace}:

``If there were no poles on the imaginary axis, then $e^{-\zeta'_{\Delta}(0)}$ would be the regularized product of eigenvalues or the Fredholm determinant of $\Delta$." 

Moreover, in \cite{karl} a connection has been established between the regularized determinant and the determinant of the discrete Laplacians in the case of Euclidean tori. Specifically, the authors of \cite{karl} obtained the formula
\begin{equation*}
\log{\det\Delta_{DT_{N(u)}}}=V(N(u))\mathcal{I}_d(0)+\log{u^2}+\log{\det\Delta_{RT,A}}+o(1) \text{ as } u\rightarrow \infty,
\end{equation*}
where $RT,A$ is the real torus $A \mathbb{Z}^d / \mathbb{R}^d$ and $DT,A$ the discrete tori while $\mathcal{I}$ is a specific special function. Inspired by the work of \cite{karl}, in \cite{chen2018regularized} a similar connection has been established in the fractal case for the $N$-dimensional double Sierpi\'nski gasket and the double $pq$-model. The goal of this paper is to expand on the work of \cite{chen2018regularized}, generalize the result to a wider class of fractal sets and shed perhaps some more light as to why the calculations in \cite{chen2018regularized} appear to work exactly as they do giving rise to such a connection between the discrete and regularized determinants. 

Our results are that under certain assumptions the spectral zeta function of the Laplacian on certain fractafolds has no poles on the imaginary axis and there also exists a connection between the logarithm of the regularized determinant and the logarithm of the discrete ones of the form
$$\log \det \Delta_n = c\,m^n+nj\log{\lambda} +\log\det\Delta$$
resembling that of \cite{karl}. We arrive at some general formulas for the determinants when establishing this connection and at the end of the paper we present a series of examples where by applying the proper values to our formulas for each specific fractal we show how we can recreate the results in \cite{chen2018regularized} by applying these new more general formulas. We then also present another example which to our knowledge does not exist in the literature showing how our results can also be used even in fractals that are not technically self-similar as long as our assumptions about spectral decimation hold.

We will be using the analysis on fractals perspective studied in \cite{kigami2001analysis, strichartz2006differential} and their references and we will mostly study self-similar fractafolds such as those in \cite{fractafold}. If we have a compact connected metric space $(X,d)$, and injective contractions $F_i: X \rightarrow X$, $i \in \{1,2,..., m\}$, then there exists a unique non-empty compact set $K \subset X$ such that 
$$K=\bigcup_{i=1}^m F_i(K).$$
This is called a self-similar set and from now on we will restrict our attention only to post-critically finite ones. A fractafold $F$ based on the self-similar set $K$ is a connected Hausdorff topological space such that every point of it has a neighborhood homeomorphic to a neighborhood in $K$. We refer the reader to \cite{fractafold} for more details. An example is the double cover of the Sierpi\'nski gasket, where we take two copies of the Sierpi\'nski gasket and glue them at the boundary points.

We can approximate a p.c.f. self-similar set with a sequence of graphs $\{G_n : n\geq 0\}$, starting from the complete graph on the boundary vertices and then taking copies of it with appropriate vertex identifications and continuing recursively at every level. In these sequence of approximating graphs we can study the discrete graph Laplacians given by $\Delta_n f(x) = \frac{1}{deg(x)}\sum_{y\sim x} (f(x) - f(y))$ where the summation is over all neighboring vertices in the n'th level graph approximation. This is the probabilistic graph Laplacian but of course there are other versions such as the combinatorial one or assigning weights at each edge.

On the other hand, it is also possible to define the Laplace operator on the fractal $K$ itself. It can be given as a limit of the discrete Laplacians after a normalization with a time--scaling factor $\lambda >1$
$$\Delta  = \lim_{n \to \infty} \lambda^n \Delta_n$$
and there is another equivalent way of defining it through weak integration via a measure, usually the Hausdorff one, and the existence of an energy form.

For a differential non-negative self-adjoint operator with compact resolvent, such as the Laplace operator with Neumann or Dirichlet conditions, its spectrum is discrete and consists of discrete eigenvalues written as
$$0 \leqslant \lambda_1 \leq \lambda_2 \leqslant \lambda_3 \leqslant \dots$$
with $\lambda_n \rightarrow \infty$. For such an operator $\Delta$ with discrete spectrum, its spectral zeta function is defined to be
$$\zeta_{\Delta}(s)=\text{Tr}\left\{\frac{1}{\Delta^s}\right\}=\sum_n \frac{1}{\lambda_n^s},$$
where the zero eigenvalue is excluded if it exists and eigenvalues are counted along with their multiplicities. Alternatively, it can be defined via Mellin transform through the heat kernel trace $K(t)$ as
$$\zeta_{\Delta}(s)=\frac{1}{\Gamma(s)}\int_0^{\infty} (K(t)-1) t^{s-1} dt$$
with the minus term being to remove the zero eigenvalue. We are interested in the product of the non-zero eigenvalues which we will abuse terminology and denote as the determinant of the operator. Formally it is 
$$\det\Delta=\prod_{i=1}^{\infty} \lambda_i$$
but of course this diverges to infinity. It is possible though to give it a real value through regularization via the spectral zeta function. Doing some formal calculations it can be seen that
$$ \zeta'_{\Delta}(0)=-\sum_{i=1}^{\infty} \log{\lambda_i}=\log{\prod_{i=1}^{\infty} \lambda_i}=-\log{\det\Delta},$$
so we can define the regularized determinant of the operator $\Delta$ to be $\det\Delta=e^{-\zeta'_{\Delta}(0)}$.

The spectrum of the Laplacian on fractals is obtained through spectral decimation. It is a process which recursively gives the eigenvalues at each level of the graph approximations from those of the previous level via an inverse image of a rational function $R$. Then for the fractal itself, we obtain the eigenvalues by taking a renormalized limit of those pre-images. Specifically, spectral decimation means that all eigenvalues of $\Delta$ are given by
\begin{equation}
\label{decimlimit}   
-\lambda^m \lim_{n\rightarrow \infty} \lambda^n R^{(-n)}(w)
\end{equation}
for $w$ in some fixed finite set and $\lambda>1$ is the time-scaling factor. The limit needs to exist and thus the branches of the pre-images $R^{-(n)}(w)$ are taken in such a way that it does exist. We will only study the cases here where the rational function turns out to be a polynomial. 

For the graph sequence $G_n$ approximating a self-similar fractal set, the number of spanning trees $\tau(G_n)$ has been studied in \cite{teufl2007enumeration, TeWa11} and in their references there. Using a different methodology so as to apply Kirchhoff's Matrix-Tree theorem for those fractals that spectral decimation holds the determinant of the graph Laplacians has been calculated in \cite{anema-tsougkas2016}
\begin{equation} 
{\label{eq:discretedet}}
\det\Delta_n=\left(\prod_{w \in A}w^{mult_n(w)}\right)\left[\prod_{w \in B}\left(w^{\sum_ {k=0}^{n}{mult_n^k(w)}}\left(\frac{-1}{a_d}\right)^{\sum_{k=0}^n mult_n^k(w)\left(\frac{d^k-1}{d-1}\right)}\right)\right]
\end{equation}
 where $d$ is the degree of the rational function and $a_d$ is the leading coefficient of the polynomial in its numerator and $mult_n^k(w)=mult_nR^{-k}(w)$ with $mult_n(w)$ being the multiplicity of the eigenvalue $w$ of $\Delta_n$. From now on we will be denoting the discrete graph Laplacians as $\Delta_n$ and we will mostly be referring to the probabilistic ones but the same results would apply if the combinatorial ones were used in the cases that spectral decimation holds. Moreover, in  \cite{LY05} the \emph{asymptotic complexity constant} of the graph $G$ was introduced, defined as
\begin{equation}
\label{eq:asympconst}    
c=\lim_{n \rightarrow \infty} \frac{\log{\tau(G_n)}}{|V_n|}
\end{equation}
if that limit exists and this constant will appear also in our work here.

Let $R(z)=a_dx^d+\dots+\lambda x$ be a polynomial with coefficients that are real numbers and $d\geq 2$ which also satisfies $R(0)=0$ and $R'(0)= \lambda > 1$ which is the case for the spectral decimation polynomial. We call $\Phi$ the entire function that is a solution to the following functional equation 
$$\Phi (\lambda z)=R(\Phi (z)) \text{ with } \Phi(0)=0, \, \Phi ' (0)=1.$$
Under the above, we can study polynomial zeta functions defined as
$$\zeta_{\Phi,w}(s)= \sum\limits_{\substack{\Phi(-\mu)=w \\ \mu >0}} \mu^{-s} \ \, \text{ or equivalently } \, \, \zeta_{\Phi,w}(s)= \lim_{n \rightarrow \infty} \sum_{z \in R^{-n}(w)} (\lambda^n z)^{-s}.$$
They appear in the calculations of the spectral zeta function $\zeta_{\Delta}(s)$ and can be meromorphically extended in $\mathbb{C}$. The following specific values are known about them. For $w<0$,
$$\zeta_{\Phi,w}(0)=0 \hspace{0.75cm} \text{ and } \hspace{0.75cm} \zeta'_{\Phi,w}(0)=-\frac{\log{a_d}}{d-1}-\log{(-w)}$$
and for $w=0$ we have that
$$\zeta_{\Phi,0}(0)=-1 \hspace{0.75cm} \text{ and } \hspace{0.75cm} \zeta'_{\Phi,0}(0)=-\frac{\log{a_d}}{d-1}.$$
Their poles are all simple and lie on the imaginary line ${\rm Re}(s)=\frac{\log{d}}{\log{\lambda}}$. Further information may be found at \cite{grabner, derfel2012laplace,teplyaev} where these zeta functions have been originally defined and studied. These polynomial zeta functions and their values will be crucial to our results.

\section{Main}

In \cite{bajorin2008vibration, malozemov2003self} spectral decimation was obtained for self-similar sets satisfying the full symmetry assumption. This procedure describes in complete detail the spectrum of the probabilistic graph Laplacian. We refer the reader to \cite{bajorin2008vibration, malozemov2003self} for more details regarding the spectral decimation process. We will first be looking at the cases where $\frac{\log d}{\log\lambda}< \frac{d_S}{2}$ and thus $d\neq m$ since $d_S = \frac{2\log m}{\log\lambda}$ is the value of the spectral dimension.

Spectral decimation for the sequence of approximating graphs essentially means, using the language of \cite{anema-tsougkas2016}, that we have two finite sets $A$, $B$ which are used to create all subsequent eigenvalues. All eigenvalues of the finite graphs can only possibly be of the following form: 

\begin{itemize}
\item 0 is always a simple eigenvalue.
\item The remaining eigenvalues of the initial graph $G_0$. Since $G_0$ is the complete graph on $|V_0|$ vertices then in the case of the standard probabilistic graph Laplacian those are $|V_0|-1$ copies of $\frac{|V_0|}{|V_0|-1}$.
\item Elements of the finite exceptional set.
\item Preiterates of the above under $R$.
\end{itemize}
We want to study the minimal sets out of which all the eigenvalues are created, that is the minimal sets in terms of generation of birth out of which every eigenvalue is either in the set $A$ or preiterates coming from the elements of the set $B$. The sets $A$ and $B$ can be described as in \cite{anema-tsougkas2016} as $A$ comprising of eigenvalues that taking preiterates is not allowed, that is for $ w \in A$ then $R^{-k}(w) \cap E \neq \emptyset$ for some k, and the set $B$ where taking preiterates is allowed forever in the sense that we never encounter a forbidden eigenvalue by doing so.

Note that the set $B$ can include elements of the exceptional set itself. We want these two sets $A$, $B$ to be minimal in the sense that every eigenvalue of the fractal is obtained by a limit of the form \eqref{decimlimit} and we are taking the lowest generation of birth possible that doesn't give us problems. For our purposes here we will also separate the set $B$ into two disjoint sets, $B_0,B_1$ where $B_0=R^{-1}(0)\setminus\{0\}$ are the eigenvalues coming from preiterates of $0$ and $B_1$ the remaining ones which necessarily can only be from the remaining branches coming from $G_0$ and the exceptional set. By construction, $B_0 \cap B_1 = \emptyset$ and we can think of $B_0$ as the one created by the $0$ branch and $B_1$ the remaining branches. 

For our results we will require explicit knowledge of the multiplicities of the eigenvalues and in this work we will only be looking at the following cases where the multiplicities of the eigenvalues of the discrete graph Laplacian in the graph approximations are of the following form. The set $A$ has multiplicities of exponential form after the first couple levels and so does the set $B_1$ while all the elements in $B_0$ must have the same multiplicities but now can have added an additional constant $j \in {0,1,2}$. This $j$ is due to the possible cases found in \cite{bajorin2008vibration}. Specifically we require that,

\begin{itemize}

\item For $w\in A$ we have $mult_n(w)=c_w^n$ for $n <2$ and $mult_n(w)=c_wm^n$ for $n \geq 2$.

\item For $w \in B_0$ there is a fixed $c_0$ such that $mult_n(w)=c_0m^n+j$ for $n \geq 1$ where $j\in \{0,1,2\}$.

\item For $w\in B_1$ we have $mult_n(w)=c_w^n$ for $n < 2$ and $mult_n(w)=c_wm^n$ for $n \geq 2$.

\end{itemize}

We will present at the last section of the paper a list of some examples where these conditions hold and further illustrate the use of this notation. Under those assumptions, our main result is the following.

\begin{theorem}
Let a  p.c.f. self-similar fractafold such that its spectrum can be obtained via spectral decimation and the spectrum of its graph approximations is of the form above. Then its spectral zeta function has no poles on the imaginary axis and we have that for $n > 1$
$$\log \det \Delta_n = c\,m^n+nj\log{\lambda} +\log\det\Delta$$
where $\det\Delta$ is the regularized determinant, $\lambda$ the time-scaling constant and $j\in \{0,1,2\}$.
\end{theorem}

\begin{remark}
If instead of the probabilistic graph Laplacian we consider the combinatorial one on graphs without loops or multiple edges and replacing $m^n$ with $|V_n|$ then the new adjusted $c$ will give us the asymptotic complexity constant given by \eqref{eq:asympconst}.
\end{remark}

It is interesting to consider when we can possibly have number of vertices that are of exponential form $|V_n| = Km^n$ in the graph approximations since in that case by looking at the cases of multiplicities in \cite{bajorin2008vibration} we can have multiplicies of the form discussed above. The following argument shows that having regular graphs on the double copy ensures that it has vertices of exponential form. Indeed, take a p.c.f self-similar set and create a double copy of it by gluing it at its corresponding boundary points. We know that for a graph, $\sum_{v\in V} deg(v)=2|E|$ and for the double graph $G_n$ we have that $|E_n|=|V_0|(|V_0|-1)m^n$. Since the graph approximations are $k$-regular, we get that
$$|V_n|=\frac{2|V_0|(|V_0|-1)}{k}m^n.$$

Before presenting the proof of our theorem we will start with the following lemmas which are needed for the technical part of the proof.

\begin{lemma}
\label{lem: prod}
Let $0,w_1,\dots, w_{d-1}$ be the roots of the spectral decimation polynomial $R(z)$ where $degR(z)=d$. Then we have that 
$$\prod_{i=1}^{d-1} w_i=(-1)^{d-1}\frac{\lambda}{a_d}$$
where $\lambda$ is the time-scaling constant.
\end{lemma}

\begin{proof}
By Vieta's formulas we have that the sum of all products of $d-1$ roots is equal to $(-1)^{d-1}\frac{a_1}{a_{d}}$ due to the fact that all other terms become 0 since 0 is also a root and the only term which remains is the product of the non-zero roots. We can now obtain the result from the fact that $R'(0)=\lambda$.
\end{proof}
We will also be needing a way to relate the polynomial spectral zeta functions of preiterates of an element with that of the element itself. This lemma will be key in obtaining the cancellation of the poles on the imaginary axis.
\begin{lemma}
\label{lem:preit}
We have that 
$$\sum_{v\in R^{-1}(\{w\})} \zeta_{\Phi,v}(s)=\lambda^s\zeta_{\Phi,w}(s).$$
Specifically, regarding $B_0$ as a multiset we have that 
$$\sum_{w\in B_0} \zeta_{\Phi,w}(s)=(\lambda^s-1)\zeta_{\Phi,0}(s).$$
\end{lemma}

\begin{proof}
Let $R^{-1}(\{w\})=\{ w_1,...,w_d \}$. By using the functional equation $\Phi(\lambda z)=R(\Phi(z))$ we can observe that 
$$\Phi(z)=w_i  \iff \Phi(\lambda z)=w \, \, \text{ and } \, \,  \Phi(z)\neq  w_j \, \, \, \forall \,  j\in \{1, \dots d\}\setminus\{i\}. $$
Then considering the polynomial zeta functions we have that
\begin{equation*}
\begin{split}
\zeta_{\Phi,w_i}(s)&= \sum_{\substack{\Phi(-\mu)=w_i \\ \mu >0}} \mu^{-s}=\quad \sum_{\mathclap{\substack{\Phi(-\lambda\mu)=w  \\ \Phi(-\mu) \neq w_j  \,  \forall  \, j \neq i \\ \mu >0}}} \mu^{-s}\\
&=  \sum\limits_{\substack{\Phi(-\lambda\mu)=w \\ \mu >0}} \mu^{-s} -\sum_{j\neq i}\sum\limits_{\substack{\Phi(-\mu)=w_j \\ \mu >0}} \mu^{-s}\\
&= \lambda^s\zeta_{\Phi,w}(s)-\sum_{j\neq i}\zeta_{\Phi, w_j}(s).
\end{split}
\end{equation*}
The second claim of the lemma is now immediate by the definition of $B_0$ and the fact that $0$ is always a fixed point of $R$.
\end{proof}

Because of their appearence in \eqref{eq:discretedet} we will also need the following simple geometric sums which can be immediately calculated by the fact that $mult_n^k=mult_{n-k}$ and using the multiplicities formulas from above.
\begin{lemma}
\label{lem:sums}
For $w \in B_1$ it holds that
\begin{equation*}
\begin{split}
&\sum_{k=0}^n mult_n^k(w)=c_w^0 + c_w^1 + \frac{c_w(m^{n+1}-m^2)}{m-1}\\
&\sum_{k=0}^{n} \text{mult}_n^k(w)\frac{(d^k-1)}{d-1} =
\frac{c_w}{d-1}\left(\frac{m^2d^n-dm^{n+1}}{d(d-m)} + \frac{m^2-m^{n+1}}{m-1}\right)\\
& \hspace{4cm} +c_w^1\frac{d^{n-1}-1}{d-1}+c_w^0\frac{d^n-1}{d-1}\\
&\\
&\text{and for \(w\in B_0\) it holds that}\\
%&\sum_{k=0}^{n} \text{mult}_n^k(w) = \sum_{k=0}^{n-1}(c_0m^{n-k}+j)=c_0\frac{m^{n+1}-m}{m-1}+jn\\
%&\sum_{k=0}^{n} \text{mult}_n^k(w)d^k = \sum_{k=0}^{n-1}(c_0m^{n-k}+j)d^k=c_0\frac{m(d^{n}-m^n)}{d-m}+j\frac{d^n-1}{d-1}
&\\
&\sum_{k=0}^{n} \text{mult}_n^k(w) =c_0\frac{m^{n+1}-m}{m-1}+jn\\
&\sum_{k=0}^{n} \text{mult}_n^k(w)d^k = c_0\frac{m(d^{n}-m^n)}{d-m}+j\frac{d^n-1}{d-1}
\end{split}
\end{equation*}
\end{lemma}

We are now ready to prove the main theorem. The proof will be done in the following steps. First we will give a general formula for the spectral zeta function and then observe that all the poles all the imaginary axis are cancelled. Then we will evaluate the regularized determinant and in the end evaluate the discrete determinants and observe the appearance of the logarithm of the regularized determinant in the logarithm of the discrete ones.

\begin{proof}

We have from \cite{grabner,teplyaev} that the spectral zeta function can be meropmorphically extended and is of the form 
$$\zeta_\Delta(s)=\sum_{w\in W}B_w(\lambda^{-s})\zeta_{\Phi,w}(s)$$ where $B_w$ are the generating functions of the multiplicities of eigenvalues, i.e. $B_w(x)=\sum_n mult_n(w)x^n$. We know from \cite{grabner,teplyaev} that the polynomial zeta functions $\zeta_{\Phi,w}(s)$ have no poles on the imaginary axis, so we need to prove that the geometric parts don't have either. Since $B_0 \cap B_1 = \emptyset $, we can write
$$\zeta_{\Delta}(s)=\sum_{w\in W}B_w(\lambda^{-s})\zeta_{\Phi,w}(s)=\sum_{w\in B_1}B_w(\lambda^{-s})\zeta_{\Phi,w}(s)+\sum_{w\in B_0}B_w(\lambda^{-s})\zeta_{\Phi,w}(s).$$
For $w\in B_1$, we have that $B_w(\lambda^{-s})=\sum_{k=0}^{1}c_w^k\lambda^{-ks}+\sum_{k=2}^{\infty}c_wm^k\lambda^{-ks}=\sum_{k=0}^{1}c_w^k\lambda^{-ks}+\frac{c_w(m\lambda^{-s})^{2}}{1-m\lambda^{-s}}$ and thus in the sum for $w\in B_1$ we have no poles on the imaginary axis. Now, for $w\in B_0$, we have that $B_w(\lambda^{-s})=\sum_{k=1}^{\infty}(c_0m^k+j)\lambda^{-ks}=\frac{c_0m\lambda^{-s}}{1-m\lambda^{-s}}+\frac{j}{\lambda^{s}-1}$ but we will now show that these poles on the imaginary axis are in fact cancelled independently of $j$. Indeed, by lemma \ref{lem:preit} we have that
\begin{equation*}
\begin{split}
\sum_{w\in B_0} B_w(\lambda^{-s})\zeta_{\Phi,w}(s)&=\sum_{w\in B_0}\left(\frac{c_0m\lambda^{-s}}{1-m\lambda^{-s}}+\frac{j}{\lambda^{s}-1}\right)\zeta_{\Phi,w}(s)\\
&=\left(\frac{c_0m\lambda^{-s}}{1-m\lambda^{-s}}+\frac{j}{\lambda^{s}-1}\right)(\lambda^s-1)\zeta_{\Phi,0}(s)
\end{split}
\end{equation*}
 and thus we have the cancellation of poles on the imaginary axis. This concludes the first part of the proof and we have obtained the following formula for the spectral zeta function

\begin{equation}
\begin{split}
\zeta_{\Delta}(s) =& \sum_{w\in B_1} \left(\sum_{k=0}^{1}c_w^k\lambda^{-ks}+\frac{c_w(m\lambda^{-s})^{2}}{1-m\lambda^{-s}}\right) \zeta_{\Phi,w}(s) \\
&+ \left(\frac{c_0m(\lambda^s-1)}{\lambda^s-m}+j\right)\zeta_{\Phi,0}(s).
\end{split}
\label{eq:zetagen}
\end{equation}
Now  by differentiating the spectral zeta function and using the values for the polynomial zetas at the point $s=0$ we can evaluate $-\zeta_{\Delta}'(0)$ to arrive at

\begin{equation}
\begin{split}
\log\det\Delta=&\sum_{w\in B_1}\left(\left(\sum_{k=0}^{1}c_w^k-\frac{c_wm^{2}}{m-1}\right) \left(\log{(-w)} + \frac{\log{a_d}}{d-1}\right)\right)\\
& -\frac{c_0m}{m-1}\log{\lambda}+ j\frac{\log{a_d}}{d-1}.
\end{split}
\label{eq:regdet}
\end{equation}

We have thus obtained the value of the regularized determinant and we will now show that this expression appears as a constant in the formula for the discrete determinants. We know that all eigenvalues are positive, so the minus signs in the formula \eqref{eq:discretedet} will all be canceled out anyway so we omit them. Moreover, to be consistent with the notation in \cite{grabner, derfel2012laplace} here we have used a spectral decimation polynomial that creates negative values $w<0$ but since the discrete graph Laplacian has positive eigenvalues all we have to do is replace $w$ with $-w$ and we can thus rewrite \eqref{eq:discretedet} in a decomposed form as
\begin{equation*}
\begin{split}
\det\Delta_n =& \prod_{w\in A}(-w)^{mult_n(w)}\prod_{w\in B_0} (-w)^{\sum_{k=0}^n mult_n^k(w)} \left(\frac{1}{a_d}\right)^{\sum_{k=0}^n mult_n^k(w)\left(\frac{d^k-1}{d-1}\right)} \\
&\cdot \prod_{w\in B_1} (-w)^{\sum_{k=0}^n mult_n^k(w)} \left(\frac{1}{a_d}\right)^{\sum_{k=0}^n mult_n^k(w)\left(\frac{d^k-1}{d-1}\right)}.
\end{split}
\end{equation*}
By using the fact that the set $B_0$ has cardinality equal to $d-1$ and that all elements in $B_0$ and their preiterates have the same multiplicities with one another which means that $mult_n^k(w)$ is independent of $w$, the $B_0$ part of the product becomes 
\begin{equation*}
\begin{split}
&\prod_{w \in B_0} (-w)^{\sum_{k=0}^n \text{mult}_n^k(w)} a_d^{-\sum_{k=0}^{n} \text{mult}_n^k(w) \left(\frac{d^k-1}{d-1}\right)}= \left(\prod_{w \in B_0} (-w)\right)^{\sum_{k=0}^n \text{mult}_n^k(w)} \hspace{-14mm} a_d^{ -\sum_{k=0}^{n} \text{mult}_n^k(w) (d^{k}-1)}\\
&=\left(\frac{\lambda}{a_d}\right)^{\sum_{k=0}^n \text{mult}_n^k(w)} a_d^{ -\sum_{k=0}^{n} \text{mult}_n^k(w) (d^k-1)} = \lambda^{\sum_{k=0}^n \text{mult}_n^k(w)} a_d^{ -\sum_{k=0}^{n} \text{mult}_n^k(w) d^k} 
\end{split}
\end{equation*}
where we have used lemma \ref{lem: prod} in the second equality.

For $w\in A$ and $n\geq 2$, we have that $mult_n(w)=c_wm^n$ so when taking logarithms the first term can be written as $\sum_{w\in A}c_wm^n \log{(-w)}$. Using lemma \ref{lem:sums} and combining everything we obtain that for $n\geq 2$
\begin{equation*}
\begin{split}
&\log\det\Delta_n = \sum_{w\in A} c_w m^n \log{(-w)} + \sum_{w\in B_1} (c_w^0 +c_w^1 +\frac{c_w m^{n+1}-c_w m^2}{m-1}) \log{(-w)}\\
&-\sum_{w \in B_1}\left(c_w^0 \frac{d^n-1}{d-1} + c_w^1 \frac{d^{n-1} -1}{d-1} + \frac{c_w}{d-1}\left(\frac{m^2 d^n - dm^{n+1}}{d(d-m)} + \frac{m^2-m^{n+1}}{m-1}\right) \right)\log{a_d}\\
&+(c_0\frac{m^{n+1}-m}{m-1} +jn)\log{\lambda} - (c_0\frac{m(d^n-m^n)}{d-m} +j\frac{d^n-1}{d-1})\log{a_d}
\end{split}
\end{equation*}
and by separating the constant term from the terms dependent on $n$ we can rewrite it as 
\begin{equation*}
\begin{split}
\log\det\Delta_n &= \sum_{w\in A} c_w m^n \log{(-w)} + \sum_{w\in B_1} \frac{c_w m^{n+1}}{m-1} \log{(-w)}\\
& -\sum_{w \in B_1}\left(\frac{c_w^0d^n+c_w^1d^{n-1}}{d-1} + \frac{c_w}{d-1}\left(\frac{m^2 d^n - dm^{n+1}}{d(d-m)} - \frac{m^{n+1}}{m-1}\right) \right)\log{a_d} \\
& +(c_0\frac{m^{n+1}}{m-1} +jn)\log{\lambda} - (c_0\frac{m(d^n-m^n)}{d-m} +j\frac{d^n}{d-1})\log{a_d} + \log\det\Delta.
\end{split}
\end{equation*}

This formula can now be simplified by the following observation. It was proven in \cite{teplstei} that the only poles of the spectral zeta function can lie either on the imaginary axis or on $\frac{d_S}{2}= \frac{\log m}{\log\lambda}$. On the other hand, the fact that the polynomial zeta functions $\zeta_{\Phi,w}(s)$ have their poles on the imaginary line $Re(s) = \frac{\log d}{\log\lambda}$ must necessarily imply that the zeroes of the geometric part cancel out those poles, a phenomenon which was first observed for the examples used at \cite{grabner, derfel2012laplace, teplyaev2004spectral}. Looking at formula \eqref{eq:zetagen} for $\lambda^s =d $ the geometric part being $0$ means that 
$$\sum_{w\in B_1} \left(c_w^0 + \frac{c_w^1}{d} + \frac{c_w m^2}{d(d-m)}\right) + \frac{c_0m(d-1)}{d-m} + j =0 $$
which simplifies the $\log\det\Delta_n$ formula above since all $d^n$ terms are now in fact cancelled. In the end, grouping up the $m^n$ terms gives us our result
\begin{equation}
\label{eq:logdisc}
\log \det \Delta_n = c\, m^n+nj\log{\lambda} +\log\det\Delta
\end{equation}
where 
$$c=  \sum_{w\in A} c_w \log{(-w)}  +  \sum_{w \in B_1} \frac{c_w m}{m-1}\left(\log{(-w)}+ \frac{\log{a_d}}{d-m}\right)  + \frac{c_0m \log{\lambda}}{m-1}  + \frac{c_0m\log{a_d}}{d-m}. $$
\end{proof}

A simpler scenario can also happen in some one dimensional sets where $\frac{\log d}{\log\lambda} = \frac{d_S}{2}$ and thus $d=m$ and if we have that the spectral zeta is of the form 
\begin{equation}
\label{eq:onedim}   
\zeta_{\Delta}(s) = \zetaa + \zetab .
\end{equation}
where $g_0$ is the non-zero eigenvalue in the $G_0$ graph. The multiplicities in the discrete graph approximation are now of a different form and every eigenvalue can be obtained as a preiterate of $0$ and $g_0$ with constant multiplicities and thus we lack any geometric part. Note that this need not mean that we never encounter a forbidden eigenvalue by taking the pre-image of $0$ or $g_0$ but rather that by using lemma \ref{lem:preit} we can simplify the spectral zeta function to be of the form of \eqref{eq:onedim}. The unit interval giving rise to the double unit interval by taking a separate copy of it and gluing its end points (which is just a circle) and some of its variations are such cases and we will include them in the example section. We omit here the details but such a spectral zeta function also does not have any poles on the imaginary axis and by using the special values of the polynomial zeta functions we can immediately obtain that 
\begin{equation}
\log\det\Delta = -\zeta'_{\Delta}(0) = \log g_0 +2\frac{\log a_d}{d-1}
\label{eq:zetareg}
\end{equation}
and by calculating the discrete determinant we get that
\begin{equation}
\log\det\Delta_n = -\frac{2\log a_d}{d-1}d^n + n\log\lambda + \log\det\Delta.
\label{eq:zetadisc}
\end{equation}
It would be interesting to have a more theoretical explanation as to why the logarithm of the time-scaling factor $\lambda$ appears as a coefficient of the term $n$ in all scenarios we have explored.

\section{Examples}

We will now look at some examples to illustrate the above notation and application of our formulas in greater clarity and show how previous existing explicitely calculated results in the literature can also follow from our formulas, as well as at the end present an original new example which we find interesting because the theory applies despite a lack of actual self-similarity. The first two examples fall in the case where $\frac{\log d}{\log\lambda} = \frac{d_S}{2}$ and the latter two where $\frac{\log d}{\log\lambda} < \frac{d_S}{2}$.

\subsection{The unit interval}

We will consider the familiar case of the unit interval $I$ which is often used as a familiar prototype in the theory of analysis on fractals since we already know what results we are expected to obtain due to standard analysis. We can take two copies of $I$ gluing them at the boundary points giving us the double unit interval which of course is just a circle. 

It was shown in \cite{teplyaev} that spectral decimation holds with polynomial $R(z) = 2z(2+z)$ and the spectral zeta function of the unit interval is $\zeta_{\Delta} (s) = \frac{4^s}{2}\zetaa$. Note that we have slightly modified that formula to be consistent with our definitions here since there is a slight difference between our work here and the work of \cite{teplyaev}. Namely, in the definitions between the spectral zeta function we have that our $s$ corresponds to $2s$ in \cite{teplyaev} and moreover the Laplacian $2\Delta$ was considered there. 
The Dirichlet spectrum coincides with the Neumann one, except the $0$ eigenvalue, so the spectral zeta function of the double unit interval is equal to $\zeta_{\Delta}(s) = 4^s\zetaa $ which in fact is of the form of \eqref{eq:onedim} and is also equal to
$$\zeta_{\Delta}(s) = \zeta_{\Phi,0}(s) + \zeta_{\Phi,-2}(s)$$
by applying lemma \ref{lem:preit}. We have that $\log\det\Delta = \log 8$ and for the discrete graphs on the double unit interval we can use the formulas \eqref{eq:zetareg} and \eqref{eq:zetadisc} which allow us to recreate the following standard result in the literature but using instead now an analysis on fractals perspective
$$\log\det\Delta_n = -2^{n+1} +n\log 4+\log\det\Delta$$
where $\lambda = 4$ as per \eqref{eq:zetadisc}.

\begin{remark}
We could have of course also calculated the discrete graph determinant by observing that the sequence of graphs are just the cyclic graphs on $V_n = 2^{n+1}$ vertices which have equally many spanning trees and then using the probabilistic version of Kirchhoff's Matrix Tree theorem.    
\end{remark}

\begin{remark}
It was also shown in \cite{teplyaev} that Riemann's function can be viewed from the analysis on fractals perspective and written as $\zeta(s) = \frac{1}{2}(\sqrt{2}\pi)^s\zeta_{\Phi,0}(\frac{s}{2})$.
\end{remark}

\subsection{The double pq-model}

\begin{figure}[ht]
    \centering
    \includegraphics[scale=0.85]{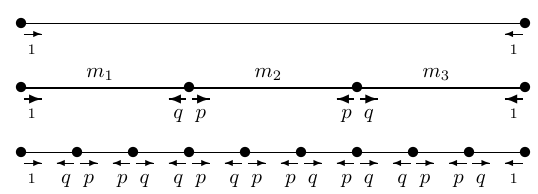}
    \caption{The random walks on the $pq$-model graphs.}
    \label{fig:pq}
\end{figure}

In \cite{teplyaev} the $pq$--model was studied which is just a different Laplacian construction on the unit interval and in Figure \ref{fig:pq} the graph approximations are depicted. We refer to \cite{teplyaev} for more details of the construction. In \cite{chen2018regularized} the double version of the $pq$-model was considered and it was shown that for the double $pq$ model we have $\sigma(\Delta_p,n) = \bigcup_{m=0}^{n-1}R_p^{-m}({0,-2})$ with spectral decimation polynomial
$$R(z)=\frac{1}{pq}(z^2+3z+2+pq)$$
and that
$$\zeta_{\Delta}(s) = \zeta_{\Phi,0}(s) + \zeta_{\Phi,-2}(s).$$
Then applying formula \eqref{eq:zetareg} and \eqref{eq:zetadisc} we obtain that $\log\det\Delta = \log\frac{2}{pq}$
and that
$$\log\det\Delta_n =  \log(pq) \, 3^n  + n\log \lambda + \log\det\Delta.$$
If we instead have applied those formulas for the combinatorial graph Laplacian, which would have slightly altered the spectral decimation polynomial and multiplied all eigenvalues by $2$, then we would have recreated exactly the result as written in \cite{chen2018regularized}. Note that the coefficient of the $n$ term as written in \cite{chen2018regularized} is indeed $\frac{(1-q^2)(1-p^2)}{(pq)^2} = 1+\frac{2}{pq} = \lambda$ in accordance with \eqref{eq:zetadisc}.

\subsection{The double Sierpi\'nski gaskets}

\begin{figure}[ht]
    \centering
    \includegraphics[scale=0.65]{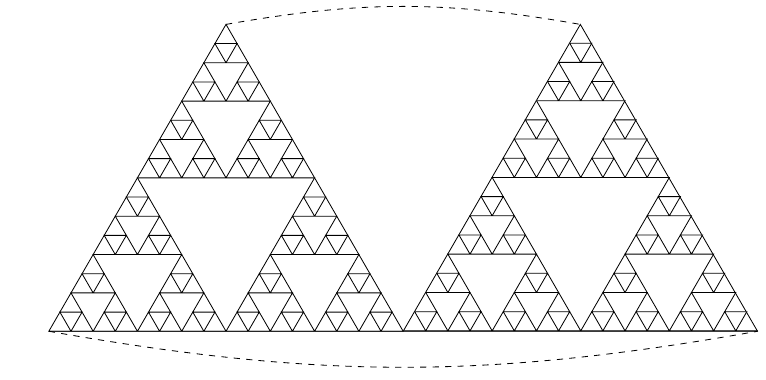}
    \caption{The double Sierpi\'nski gasket for $N=3$.}
    \label{fig:sg}
\end{figure}

We will look at the $N-1$ dimensional SG where it is obtained by starting with the complete graph on $N$ vertices. The standard SG is for $N=3$ and in \cite{chen2018regularized} the double Sierpi\'nski gasket was considered where it is obtained by taking two copies of it and gluing the corresponding boundary points, the two dimensional case is shown in Figure \ref{fig:sg}. We will show how the results of \cite{chen2018regularized} for the double Sierpi\'nski gasket can be obtained directly by applying our formulas \eqref{eq:regdet} and \eqref{eq:logdisc}. In \cite{chen2018regularized} the combinatorial graph Laplacian was studied but instead we will here apply the formula for the probabilistic one but both cases can be recreated using our formulas by using the appropriate spectral decimation polynomial and corresponding sets $A, B_0, B_1$ each time. In any case, due to the regularity of the graphs the determinant is the same up to a different factor of $(\frac{1}{2N-2})^{V_n-1}$ since each eigenvalue is scaled with the same factor and the zero eigenvalue is excluded. It was shown in \cite{fractafold} that the double $SG$ eigenvalues are those of the Dirichlet and Neumann spectra of the standard $SG$ combined. Both the Dirichlet and the Neumann spectra have been evaluated in \cite{fukushima1992spectral, shima1991eigenvalue}. We have that $|V_n|=N^{n+1}$ and the spectral decimation polynomial is $R(z)=z(N+2+(2N-2)z)$. In our language we can write that 
$$A= \{ -\frac{N}{N-1} \}, \, \, B_0= \{-\frac{N+2}{2N-2}\}, \, \, B_1=\{-\frac{1}{N-1},-\frac{N}{2N-2}\}$$
and for the multiplicities, we have that 
\begin{equation*}
\begin{split}
&mult_n(\frac{N}{N-1})=(N-2)N^m \text{ for } n\geq 1 \text{ and thus } c_{\frac{N}{N-1}}= N-2.\\
&mult_n(\frac{N}{2N-2})=\frac{(N-2)}{N}N^{n} \text{ for } n \geq 2 \text{ and } mult_1(\frac{N}{2N-2})=N-1.\\
&\text{Thus, } (c_{\frac{N}{2N-2}}^0,c_{\frac{N}{2N-2}}^1)=(0,N-1) \text{ and } c_{\frac{N}{2N-2}}=\frac{N-2}{N}.\\
& mult_n(\frac{N+2}{2N-2})=\frac{(N-2)}{N}N^{n}+1 \text{ for } n\geq 1 \\
&\text{and therefore }c_{0}=\frac{N-2}{N} \text{ and } j=1 \text{ with } c_{\frac{N+2}{N}}^0 = N-1 .\\
&mult_n(\frac{1}{N-1})=0 \text{ for } n \geq 2 \text{ and } (c_{\frac{1}{N-1}}^0,c_{\frac{1}{N-1}}^1)=(0,1) \text{ and } c_{\frac{1}{N-1}}=0.
\end{split}
\end{equation*}
This gives us that according to \eqref{eq:regdet} and after some calculations that
$$\log\det\Delta = 2\log{2} + \frac{\log{N}}{N-1} - \frac{N-2}{N-1}\log{(N+2)} + \log{(N-1)}$$
and  by \eqref{eq:logdisc} that
% \begin{equation*}
% \begin{split}
% \log\det \Delta_n =& -2N^n\log{2} +\frac{(N-2)N^{n+1}}{N-1}\log{N} -N^{n+1} \log{(N-1)}\\
% &+\left( \frac{(N-2)N^n}{N-1} + n \right)\log{(N+2)} + \log\det\Delta.
% \end{split}
% \end{equation*}
$$\log\det\Delta_n = c \,N^n  + n\log \lambda + \log\det\Delta$$
where
$$c=-2\log 2 +\frac{N(N-2)}{N-1}\log N -N\log{(N-1)}+\frac{N-2}{N-1}\log{(N+2)}.$$

\begin{remark}
If we have used the combinatorial graph Laplacian instead then the value $c$ would have been the asymptotic complexity constant of those graphs which coincides with that of the single Sierpi\'nski gaskets.    
\end{remark}

\subsection{The Basilica set}

We will now give a last and original example which is not self-similar in the strict sense but spectral decimation still holds as well as our assumptions so our theorem holds. Specifically, we look at the Basilica Julia set from \cite{rogers2008laplacians}, shown in Figure \ref{fig:basil} where the spectrum was calculated but this connection between the logarithms of the regularized determinant and the discrete one has not been previously observed in the literature to our knowledge. It was shown in \cite{rogers2008laplacians} that $R(z)=z(\frac{2p+1}{p}+\frac{1}{p}z)$ with $m=3$ which is not technically the number of self-similarities since this example is not actually self-similar. Taking from \cite{rogers2008laplacians} the spectrum information we need here we have that $|V_n|=2\cdot 3^n$ and $a_d = \frac{1}{p}, d=2, m=3, \lambda = \frac{2p+1}{p}$ and using our notation here we get $A=\emptyset, \, \, B_0= \{-2p-1\},  \, \, B_1=\{-2q,-2p\}$. For the multiplicities, we have that

\begin{figure}[t]
    \centering
    \includegraphics[scale=0.45]{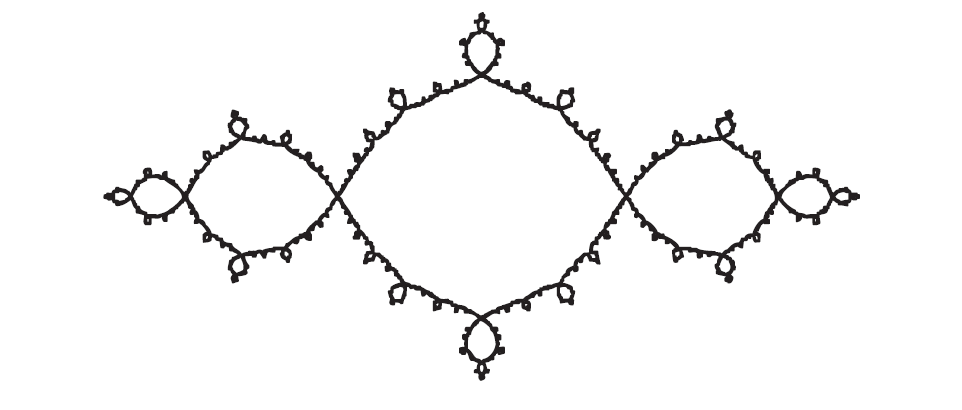}
    \caption{The Basilica set.}
    \label{fig:basil}
\end{figure}

\begin{equation*}
\begin{split}
&mult_n(2p)= 2\cdot 3^{n-1} \text{ for } n\geq 1 \text{ and thus } c_{2p}=\frac{2}{3}.\\
&mult_n(2p+1)=1 \text{ for } n \geq 1 \text{ and thus } j=1 \text{ and } c_{0}=0.\\
& mult_n(2q)=0 \text{ for } n\geq 0 \text{ and thus }c_{2q}=0 \text{ and } c_{2q}^0=1.
\end{split}
\end{equation*}
We obtain the spectral zeta function given by \eqref{eq:zetagen} as follows
$$\zeta_{\Delta}(s)=\zeta_{\Phi,0}(s)+\zeta_{\Phi,-2q}(s)+\frac{2}{\lambda^s-3}\zeta_{\Phi,-2p}(s)$$
and by \eqref{eq:regdet} the regularized determinant has logarithm
$$\log\det\Delta=-\zeta_{\Delta}'(0)=\log\frac{q}{p^2}$$
and for the discrete graph Laplacian applying formula \eqref{eq:logdisc} gives us the following connection between the discrete and regularized logarithms
$$\log\det\Delta_n=c \,  3^n +n\log{\lambda}+\log\det\Delta .$$
where $c=\log{(2p^2)}$.

\section*{Acknowledgments}
The author is grateful to Alexander Teplyaev, Joe P. Chen and Anders Karlsson for helpful discussions, as well as the Mathematics Department of the University of Connecticut and University of Uppsala where part of this research was conducted at.

%\nocite{*}
\bibliographystyle{abbrv}
\bibliography{Biblio}

\begin{thebibliography}{10}

\bibitem{akkermans2013statistical}
E.~Akkermans.
\newblock Statistical mechanics and quantum fields on fractals.
\newblock {\em Fractal Geometry and Dynamical Systems in Pure and Applied Mathematics II: Fractals in Applied Mathematics}, 601:1, 2013.

\bibitem{akkermans2012wave}
E.~Akkermans, G.~Dunne, and E.~Levy.
\newblock Wave propagation in one-dimension: Methods and applications to complex and fractal structures.
\newblock {\em arXiv preprint arXiv:1210.7409}, 2012.

\bibitem{akkermans2009physical}
E.~Akkermans, G.~V. Dunne, and A.~Teplyaev.
\newblock Physical consequences of complex dimensions of fractals.
\newblock {\em EPL (Europhysics Letters)}, 88(4):40007, 2009.

\bibitem{akkermans2010thermodynamics}
E.~Akkermans, G.~V. Dunne, and A.~Teplyaev.
\newblock Thermodynamics of photons on fractals.
\newblock {\em Physical review letters}, 105(23):230407, 2010.

\bibitem{akkermans2013spontaneous}
E.~Akkermans and E.~Gurevich.
\newblock Spontaneous emission from a fractal vacuum.
\newblock {\em EPL (Europhysics Letters)}, 103(3):30009, 2013.

\bibitem{anema-tsougkas2016}
J.~A. Anema and K.~Tsougkas.
\newblock Counting spanning trees on fractal graphs and their asymptotic complexity.
\newblock {\em Journal of Physics A: Mathematical and Theoretical}, 49(35):355101, 2016.

\bibitem{bajorin2008vibration}
N.~Bajorin, T.~Chen, A.~Dagan, C.~Emmons, M.~Hussein, M.~Khalil, P.~Mody, B.~Steinhurst, and A.~Teplyaev.
\newblock Vibration modes of 3n-gaskets and other fractals.
\newblock {\em Journal of Physics A: Mathematical and Theoretical}, 41(1):015101, 2008.

\bibitem{chen2015spectral}
J.~P. Chen, S.~Molchanov, and A.~Teplyaev.
\newblock Spectral dimension and {B}ohr's formula for {S}chr{\"o}dinger operators on unbounded fractal spaces.
\newblock {\em Journal of Physics A: Mathematical and Theoretical}, 48(39):395203, 2015.

\bibitem{chen2016singularly}
J.~P. Chen and A.~Teplyaev.
\newblock Singularly continuous spectrum of a self-similar {L}aplacian on the half-line.
\newblock {\em Journal of Mathematical Physics}, 57(5):052104, 2016.

\bibitem{chen2018regularized}
J.~P. Chen, A.~Teplyaev, and K.~Tsougkas.
\newblock Regularized laplacian determinants of self-similar fractals.
\newblock {\em Letters in mathematical physics}, 108(6):1563--1579, 2018.

\bibitem{karl}
G.~Chinta, J.~Jorgenson, and A.~Karlsson.
\newblock Zeta functions, heat kernels, and spectral asymptotics on degenerating families of discrete tori.
\newblock {\em Nagoya mathematical journal}, 198:121--172, 2010.

\bibitem{grabner}
G.~Derfel, P.~Grabner, and F.~Vogl.
\newblock The zeta function of the {L}aplacian on certain fractals.
\newblock {\em Transactions of the American Mathematical Society}, 360(2):881--897, 2008.

\bibitem{derfel2012laplace}
G.~Derfel, P.~J. Grabner, and F.~Vogl.
\newblock Laplace operators on fractals and related functional equations.
\newblock {\em Journal of Physics A: Mathematical and Theoretical}, 45(46):463001, 2012.

\bibitem{dunne2012heat}
G.~V. Dunne.
\newblock Heat kernels and zeta functions on fractals.
\newblock {\em Journal of Physics A: Mathematical and Theoretical}, 45(37):374016, 2012.

\bibitem{elizalde2012introduction}
E.~Elizalde.
\newblock Introduction and outlook.
\newblock In {\em Ten Physical Applications of Spectral Zeta Functions}, pages 1--22. Springer, 2012.

\bibitem{elizalde1994zeta}
E.~Elizalde, S.~Odintsov, A.~Romeo, A.~A. Bytsenko, and S.~Zerbini.
\newblock {\em Zeta regularization techniques with applications}.
\newblock World Scientific, 1994.

\bibitem{englert1987metric}
F.~Englert, J.-M. Fr{\`e}re, M.~Rooman, and P.~Spindel.
\newblock Metric space-time as fixed point of the renormalization group equations on fractal structures.
\newblock {\em Nuclear Physics B}, 280:147--180, 1987.

\bibitem{freiberg2007einstein}
U.~R. Freiberg.
\newblock Einstein relation on fractal objects.
\newblock In {\em Communications to SIMAI Congress}, volume~2, 2007.

\bibitem{fukushima1992spectral}
M.~Fukushima and T.~Shima.
\newblock On a spectral analysis for the {S}ierpinski gasket.
\newblock {\em Potential Analysis}, 1(1):1--35, 1992.

\bibitem{hawking1977zeta}
S.~W. Hawking.
\newblock Zeta function regularization of path integrals in curved spacetime.
\newblock {\em Communications in Mathematical Physics}, 55(2):133--148, 1977.

\bibitem{kigami2001analysis}
J.~Kigami.
\newblock {\em Analysis on fractals}, volume 143.
\newblock Cambridge University Press, 2001.

\bibitem{kigami1993weyl}
J.~Kigami and M.~L. Lapidus.
\newblock Weyl's problem for the spectral distribution of {L}aplacians on pcf self-similar fractals.
\newblock {\em Communications in mathematical physics}, 158(1):93--125, 1993.

\bibitem{knizhnik1988fractal}
V.~G. Knizhnik, A.~M. Polyakov, and A.~B. Zamolodchikov.
\newblock Fractal structure of 2d—quantum gravity.
\newblock {\em Modern Physics Letters A}, 3(08):819--826, 1988.

\bibitem{lapidus2000fractal}
M.~L. Lapidus and M.~Van~Frankenhuysen.
\newblock {\em Fractal Geometry and Number Theory: Complex dimensions of fractal strings and zeros of zeta functions}.
\newblock Cambridge Univ Press, 2000.

\bibitem{lauscher2005fractal}
O.~Lauscher and M.~Reuter.
\newblock Fractal spacetime structure in asymptotically safe gravity.
\newblock {\em Journal of High Energy Physics}, 2005(10):050, 2005.

\bibitem{LY05}
R.~Lyons.
\newblock Asymptotic enumeration of spanning trees.
\newblock {\em Combin. Probab. Comput.}, 14(4):491--522, 2005.

\bibitem{malozemov2003self}
L.~Malozemov and A.~Teplyaev.
\newblock Self-similarity, operators and dynamics.
\newblock {\em Mathematical Physics, Analysis and Geometry}, 6:201--218, 2003.

\bibitem{reuter2011fractal}
M.~Reuter and F.~Saueressig.
\newblock Fractal space-times under the microscope: a renormalization group view on monte carlo data.
\newblock {\em Journal of High Energy Physics}, 2011(12):1--27, 2011.

\bibitem{rogers2008laplacians}
L.~G. Rogers and A.~Teplyaev.
\newblock Laplacians on the basilica julia set.
\newblock {\em arXiv preprint arXiv:0802.3248}, 2008.

\bibitem{shima1991eigenvalue}
T.~Shima.
\newblock On eigenvalue problems for the random walks on the {S}ierpinski pre-gaskets.
\newblock {\em Japan Journal of Industrial and Applied Mathematics}, 8(1):127--141, 1991.

\bibitem{shima1996eigenvalue}
T.~Shima.
\newblock On eigenvalue problems for {L}aplacians on pcf self-similar sets.
\newblock {\em Japan Journal of Industrial and Applied Mathematics}, 13(1):1--23, 1996.

\bibitem{teplstei}
B.~A. Steinhurst and A.~Teplyaev.
\newblock Existence of a meromorphic extension of spectral zeta functions on fractals.
\newblock {\em Letters in Mathematical Physics}, 103(12):1377--1388, 2013.

\bibitem{fractafold}
R.~Strichartz.
\newblock Fractafolds based on the {S}ierpinski gasket and their spectra.
\newblock {\em Transactions of the American Mathematical Society}, 355(10):4019--4043, 2003.

\bibitem{strichartz2012exact}
R.~Strichartz.
\newblock Exact spectral asymptotics on the {S}ierpinski gasket.
\newblock {\em Proceedings of the American Mathematical Society}, 140(5):1749--1755, 2012.

\bibitem{strichartz2006differential}
R.~S. Strichartz.
\newblock {\em Differential equations on fractals: a tutorial}.
\newblock Princeton University Press, 2006.

\bibitem{strichartz2012spectral}
R.~S. Strichartz and A.~Teplyaev.
\newblock Spectral analysis on infinite {S}ierpi{\'n}ski fractafolds.
\newblock {\em Journal d'Analyse Math{\'e}matique}, 116(1):255--297, 2012.

\bibitem{tanese2014fractal}
D.~Tanese, E.~Gurevich, F.~Baboux, T.~Jacqmin, A.~Lema{\^\i}tre, E.~Galopin, I.~Sagnes, A.~Amo, J.~Bloch, and E.~Akkermans.
\newblock Fractal energy spectrum of a polariton gas in a {F}ibonacci quasiperiodic potential.
\newblock {\em Physical review letters}, 112(14):146404, 2014.

\bibitem{teplyaev2004spectral}
A.~Teplyaev.
\newblock Spectral zeta function of symmetric fractals.
\newblock In {\em Fractal geometry and stochastics III}, pages 245--262. Springer, 2004.

\bibitem{teplyaev}
A.~Teplyaev.
\newblock Spectral zeta functions of fractals and the complex dynamics of polynomials.
\newblock {\em Transactions of the American Mathematical Society}, 359(9):4339--4358, 2007.

\bibitem{teufl2007enumeration}
E.~Teufl and S.~Wagner.
\newblock Enumeration problems for classes of self-similar graphs.
\newblock {\em Journal of Combinatorial Theory, Series A}, 114(7):1254--1277, 2007.

\bibitem{TeWa11}
E.~Teufl and S.~Wagner.
\newblock The number of spanning trees in self-similar graphs.
\newblock {\em Ann. Comb.}, 15(2):355--380, 2011.

\end{thebibliography}

\end{document}